\documentclass[preprint,3p,sort]{elsarticle}

\usepackage{amsfonts}
\usepackage{amsthm}
\usepackage{graphics}
\usepackage{epstopdf}
\usepackage{algorithmic}
\usepackage{ams math}
\usepackage{amssymb}
\usepackage{hyperref}
\usepackage[T1]{fontenc}
\usepackage{soul,color}
\usepackage{cases}
\newtheorem{theorem}{Theorem}[section]
\newtheorem{proposition}[theorem]{Proposition}
\newtheorem{lemma}[theorem]{Lemma}

\newtheorem{remark}[theorem]{Remark}
\theoremstyle{example}
\newtheorem{example}[theorem]{Example}
\numberwithin{equation}{section}
\begin{document}


\begin{frontmatter}

%

\title{Linear and Nonlinear Boundary Conditions: What's the difference? }

\author[sweden,southafrica]{Jan Nordstr\"{o}m}
\cortext[secondcorrespondingauthor]{Corresponding author}
\ead{jan.nordstrom@liu.se}
\address[sweden]{Department of Mathematics, Applied Mathematics, Link\"{o}ping University, SE-581 83 Link\"{o}ping, Sweden}
\address[southafrica]{Department of Mathematics and Applied Mathematics, University of Johannesburg, P.O. Box 524, Auckland Park 2006, Johannesburg, South Africa}

\begin{abstract}
In previous work, we derived new energy and entropy stable open boundary conditions and implementation procedures for linear and nonlinear initial boundary
value problems. These boundary procedures results in estimates bounded by external data only. 
Interestingly, these new boundary conditions generalize the well-known classical characteristic boundary conditions for linear problems to the nonlinear setting. We discuss the similarities and differences between these two boundary procedures and point out the advantages with the new procedures. In particular we show that the new boundary conditions  bound both linear and nonlinear initial boundary value problems and can be implemented both strongly and weakly.
\end{abstract}



\end{frontmatter}


\section{Introduction}
Energy and entropy stability for discretized initial boundary value problems (IBVPs) require boundary procedures at open boundaries for the continuous IBVPs that bounds the surface terms in external data. 
For linear hyperbolic IBVPs, a common boundary procedure is based on eigenvalue and eigenvector decomposition of the boundary term where one specifies the ingoing characteristic variables in terms  of the boundary data (and possibly outgoing ones) \cite{kreiss1970,Majda1975607,Higdon}. Characteristic based boundary conditions have also been used (although more rarely) in nonlinear \cite{THOMPSON19871,HEDSTROM1979222}  settings.
A new type of  characteristic like boundary procedures for open boundaries suitable for {\it both} linear and nonlinear  IBVPs were recently derived  \cite{nordstrom2022linear-nonlinear,NORDSTROM2024_BC,Nordström2025_open}. 
We discuss the relation between these new nonlinear characteristic  boundary conditions and the classical linear characteristic boundary conditions and point out the generality and advantages with the new procedures.

\section{Analysis of a scalar IBVP}\label{sec:illustration_theory}
To introduce the main idea in a simple setting, we consider the following scalar skew-symmetric IBVP
\begin{equation}\label{eq:nonlin_il}
u_t + (a u)_x+au_x=0,  \quad t \geq 0,  \quad  x  \geq 0, 
\end{equation}
with the initial condition $u(x,0)=f(x)$ and  the non-homogeneous linear or nonlinear boundary condition
\begin{equation}\label{eq:nonlin_BC_il}
b(u) -  g(t)=0,  \quad t \geq 0,  \quad  x=0.
\end{equation}
In (\ref{eq:nonlin_il}),  $a>0$ is a possibly solution dependent function. 
The initial data $f$ and the boundary data $g$
constitute input to the IBVP. 
We only consider the boundary $x=0$ and $f$ has compact support.
 \begin{remark}\label{explain_scal_energy_entropy} 
With a slight  abuse of notation we denote both energy and mathematical entropy as "energy".
 \end{remark}

By applying the energy method (multiply the equation with $u$ and integrate-by-parts) to (\ref{eq:nonlin_il}) we
obtain
\begin{equation}\label{eq:boundaryPart1_il}
\frac{1}{2} \frac{d}{dt}\|u\|^2 - (a u^2)_{x=0}=0,
\end{equation}
where $\|u\|^2= \int_0^{\infty} u^2 dx$. The skew-symmetry of (\ref{eq:nonlin_il}) leaves the energy rate to be determined by the boundary terms only. 
Note that specifying $u$ in (\ref{eq:boundaryPart1_il}) does not necessarily bound the energy, also the wave speed $a$ matter, i.e. 
 the complete boundary term must be considered when constructing the boundary operator $b(u)$. 
\subsection{Boundary conditions and implementation procedures}\label{BC_theory_il} 
We will consider three types of boundary conditions (\ref{eq:nonlin_BC_il}) for equation (\ref{eq:nonlin_il}). Firstly $b(u)=u$ used in standard advection problems, secondly the flux $b(u)=au$ typically used for conservation laws and finally the non-conventional $b(u)=\sqrt{a}u$ as discussed in \cite{nordstrom2022linear-nonlinear,NORDSTROM2024_BC}.  We start by imposing the boundary conditions strongly.
\subsubsection{Strong implementation}\label{BC_theory_il_strong} 
A strong implementation of the first condition $u-g=0$ gives the energy rate 
\begin{equation}\label{eq:scal1}
\frac{1}{2} \frac{d}{dt}\|u\|^2 =ag^2
\end{equation}
which is not bounded when $a=a(u,x,t)$. Secondly, the strongly imposed flux like condition $au-g$ yields 
\begin{equation}\label{eq:scal2}
\frac{1}{2} \frac{d}{dt}\|u\|^2 =ug
\end{equation}
which is bounded only for $g_2=0$. Finally the non-conventional choice $\sqrt{a}u-g=0$ yields the perfect result
\begin{equation}\label{eq:scal3}
\frac{1}{2} \frac{d}{dt}\|u\|^2 =g^2,
\end{equation}
and the possible growth due to the nonlinearity caused by $a=a(u,x,t)$ is now controlled in (\ref{eq:scal3}).
\begin{remark}\label{explain_scal} 
Note that the boundary condition $\sqrt{a}u-g=0$  leading up to the energy rate (\ref{eq:scal3}) is suitable for both linear and nonlinear settings. The scaling with $\sqrt{a}$ makes it possible to square the boundary term.
 \end{remark}
\subsubsection{Weak  implementation}\label{BC_theory_il_strong} 
Following \cite{nordstrom_roadmap}, we derive the conditions for energy boundedness in the continuous setting
such that we can straightforwardly mimic that numerically and construct energy stable so-called SBP-SAT schemes  \cite{svard2014review,fernandez2014review}. 
Equation (\ref{eq:nonlin_il}) augmented with a lifting operator for implementing the boundary condition (\ref{eq:nonlin_BC_il}) is 
\begin{equation}\label{eq:nonlin_lif_il}
u_t + (a u)_x+au_x + l(2\sigma(b(u)-g))=0,  \quad t \geq 0,  \quad  x  \geq 0,
\end{equation}
where $\sigma$ is a yet undetermined penalty parameter. The lifting operator $l$ for two smooth functions  $\phi, \psi$ satisfies $\int\limits \phi   l(\psi) dx = (\phi  \psi)_{x=0}$. The lifting operator enables development of the essential parts of the weak numerical boundary procedure using SAT terms already in the continuous setting \cite{Arnold20011749,nordstrom_roadmap}. 
By once more applying the energy method, now to (\ref{eq:nonlin_lif_il}),  we find that the energy rate becomes 
\begin{equation}\label{eq:boundaryPart1_il_pen}
\frac{1}{2} \frac{d}{dt}\|u\|^2 - a u^2+2\sigma u(b(u)-g)=0.
\end{equation}


Since the only boundary conditions that lead to an estimate for  strong imposition was the non-conventional boundary condition $\sqrt{a}u-s g=0$, we restrict the analysis to that case. With $\sigma = \sqrt{a}$ we find
\begin{equation}\label{eq:boundaryPart1_il_pen_nonlin}
\frac{1}{2} \frac{d}{dt}\|u\|^2 = a u^2-2 \sqrt{a}  u(\sqrt{a}u-g)=-(\sqrt{a} u)^2 + 2 (\sqrt{a} u) g=g^2 - (\sqrt{a}u-g)^2
\end{equation}
where we added and subtracted $g^2$. This gives a bound in both the linear and nonlinear case. As for the strong imposition, the possible growth due to the nonlinearity caused by $a=a(u,x,t)$ is now controlled.
\begin{remark}\label{explain_4} 
In the upcoming analysis for nonlinear systems of equations we have boundary terms of the form $U^T A(U) U$. 
Controlling the coefficient $a(u)$
corresponds to controlling 
the matrix $A(U)$.
\end{remark}

\section{Analysis of IBVPs for systems of equations}\label{sec:theory}
We recapitulate  some results from \cite{nordstrom2022linear-nonlinear,Nordstrom2022_Skew_Euler,NORDSTROM2024_BC,Nordström2024,Nordström2025_VOF,Nordström2025_open} to set the stage for the boundary condition analysis.  
 \subsection{The energy rate}\label{energy_rate} 
Consider the following skew-symmetric IBVP
\begin{equation}\label{eq:nonlin}
P U_t + (A_i U)_{x_i}+A_i^TU_{x_i}= 0,  \quad t \geq 0,  \quad  \vec x=(x_1,..,x_k) \in \Omega
\end{equation}
augmented with the initial condition $U(\vec x,0)=F(\vec x)$ in $\Omega$ and  the non-homogeneous boundary condition
\begin{equation}\label{eq:nonlin_BC}
B(U) - G(\vec x,t)=0,  \quad t \geq 0,  \quad  \vec x=(x_1,..,x_k) \in  \partial\Omega.
\end{equation}
In (\ref{eq:nonlin}), $U$ is a smooth  $n$ vector, the $n \times n $ matrices $A_i$ are smooth, Einsteins summation convention is used and $P$ is a symmetric positive definite (or semi-definite) time-independent matrix that defines an energy norm (or semi-norm) $\|U\|^2_P= \int_{\Omega} U^T P U d\Omega$. 
In (\ref{eq:nonlin_BC}), $B$ is the linear or nonlinear boundary operator and $G$ the boundary data. 
Note that (\ref{eq:nonlin}) and (\ref{eq:nonlin_BC}) encapsulates both linear 
and nonlinear  ($A_i=A_i(U)$) problems. 

The energy method  (multiply with $U^T$ and integrate over domain) applied to (\ref{eq:nonlin}) yields
\begin{equation}\label{eq:boundaryPart1_1}
\frac{1}{2} \frac{d}{dt}\|U\|^2_P + \oint\limits_{\partial\Omega}U^T  (n_i A_i)  U  ds= 0
\end{equation}
where $(n_1,..,n_k)^T$ is the outward pointing unit normal.  Note that the energy rate (\ref{eq:boundaryPart1_1}) was obtained without knowledge of details in the matrices $A_i$, only the skew-symmetry of (\ref{eq:nonlin}) was required  \cite{nordstrom2022linear-nonlinear,Nordstrom2022_Skew_Euler,NORDSTROM2024_BC,Nordström2024,Nordström2025_VOF,Nordström2025_open}.
\begin{remark}\label{explain_min_no} 
A minimal number of boundary conditions that leads to a bound in the linear case, also  leads to uniqueness 
\cite{kreiss1970,Strikwerda1977797,nordstrom2020}. 
For linear IBVPs,  the number of boundary conditions depends on external data, i.e.~the eigenvalues of $n_i A_i$. For nonlinear problems, the eigenvalues are solution dependent and it it is not completely known what that implies, see \cite{nordstrom2022linear} for examples. A minimal number of boundary conditions that lead to a bound is a necessary but maybe not a sufficient condition for uniqueness in the nonlinear case.
\end{remark}
 \subsection{Boundary conditions and implementation procedures}\label{BC_theory} 
We start by observing that the matrix in (\ref{eq:boundaryPart1_1}) is symmetric since $ U^T  (n_i A_i)   U = U^T ((n_i A_i)  +(n_i A_i )^T) U /2= U^T A U$. 
Next, the matrix $A$ is transformed to diagonal form as $ T^T A  T =  \Lambda =  \Lambda^+ +  \Lambda^-$
with new characteristic variables $W = T^{-1} U=T^T U$. The split of eigenvalues into positive and negative parts implies
\begin{equation}\label{1Dprimalstab_trans_final}
U^T   A  U   = W^T   \Lambda   W = W^T( \Lambda^+ +  \Lambda^-) W=W_+^T   \Lambda^+  W_+ + W_-^T   \Lambda^-  W_-=W_+^T   \Lambda^+ W_+ - W_-^T  |\Lambda^-| W_-.
\end{equation}
In (\ref{1Dprimalstab_trans_final}), the indicator matrices  $I^-, I^+$ where $I^-+I^+=I$ define $\Lambda^+ = I^+ \Lambda$ and  $\Lambda^-=  I^- \Lambda$ 
, $W^+=I^+ W$  and $W^- = I^- W$. The notation $A^+ = T \Lambda^+ T^T$,  $A^- = T \Lambda^- T^T$ and $|A^-|= T  |\Lambda^-| T^T$ leads to 
\begin{equation}\label{1Dprimalstab_trans_final_standard}
U^T   A  U    = U^T( A^+ +  A^-) U=U^T  A^+ U + U^T  A^- U=U^T  A^+ U - U^T  |A^-| U.
\end{equation}
In addition, we will also need the somewhat non-conventional but practical notation 
\begin{equation}\label{1Dprimalstab_trans_final_nonstandard}
A^+ =T \Lambda^+ T^T=T \sqrt{\Lambda^+}  T^T T \sqrt{\Lambda^+}  T^T =  \sqrt{A^+}  \sqrt{A^+}   \quad  \text{and similary} \quad |A^-|= \sqrt{ |A^-|}  \sqrt{ |A^-|}.
\end{equation}
\begin{remark}\label{eigenvalues} Stability problems are related to negative eigenvalues $\Lambda^-$. Possible zero eigenvalues are included in $\Lambda^+$. Hence, stabilizing boundary conditions must adjust terms related to $A^-$, $|A^-|$ and  $\sqrt{ |A^-|} $. 
\end{remark}
\begin{remark}\label{Sylvester}
If the diagonal entries $\Lambda$ are obtained by rotation where $T^T \neq T^{-1}$ \cite{nordstrom2022linear-nonlinear,Nordstrom2022_Skew_Euler,NORDSTROM2024_BC,Nordström2024,Nordström2025_VOF,Nordström2025_open}, Sylvester's Criterion \cite{horn2012}, shows that the number of diagonal entries with negative sign are equal to the number of negative eigenvalues if the matrix  $T$ is non-singular.  In the rotation case, $ \sqrt{A^+} =\sqrt{\Lambda^+} T^{-1}$ and $\sqrt{ |A^-|} =\sqrt{ |\Lambda^-|} T^{-1}$ leading to  $U^T   A  U =(\sqrt{A^+} U)^T(\sqrt{A^+}  U) - (\sqrt{ |A^-|} U)^T(\sqrt{ |A^-|} U)$ 
should be used \cite{Nordström2025_open}. 
\end{remark}

\subsubsection{Strong implementation}\label{BC_theory_il_strong_system} 
We will consider two classical characteristic boundary conditions for linear problems but now applied in nonlinear setting. They will be compared with the new non-conventional characteristic boundary condition.

Firstly we specify the ingoing characteristic variables $W^--G=0$ related to 
$\Lambda^-=- |\Lambda^-|$ and  obtain
\begin{equation}\label{1Dprimalstab_trans_final_extra_1}
U^T  A U = W^T   \Lambda W =  W_+^T   \Lambda^+ W_+  + G^T   \Lambda^-   G \geq  - G^T   |\Lambda^-|   G.
\end{equation}
The relation (\ref{1Dprimalstab_trans_final_extra_1}) gives an estimate in the linear, but not the nonlinear case (see (\ref{eq:scal1}) for the scalar estimate). 

Secondly, we 
specify the ingoing flux in (\ref{eq:nonlin_BC}) as $|A^-|U-G=0$ which gives
\begin{equation}\label{1Dprimalstab_trans_final_extra_2}
U^T   A U = U^T  A^+ U - U^T|A^-|U = U^T  A^+ U - U^T G \geq -U^T G. 
\end{equation}
As in the scalar case  (\ref{eq:scal2}), we only get an estimate for (\ref{1Dprimalstab_trans_final_extra_2}) with zero data $G=0$. 

Finally we investigate the new non-conventional choice $\sqrt{ |A^-|}  U-G=0$ corresponding to (\ref{eq:scal3}) and find
\begin{equation}\label{1Dprimalstab_trans_final_extra_3}
U^T   A  U  = U^T  A^+ U - U^T|A^-|U =  U^T  A^+ U - (\sqrt{ |A^-|}   U)^T(\sqrt{ |A^-|}   U) = U^T  A^+ U - G^T G \geq -G^TG.
\end{equation}
The estimate in (\ref {1Dprimalstab_trans_final_extra_3}) bounds and controls the potential growth. Similarly to the scalar case, the new characteristic boundary condition $\sqrt{ |A^-|}  U-G=0$ bounds both linear and nonlinear cases.

\subsubsection{Weak implementation}\label{BC_theory_il_weak_system} 
We introduce a lifting operator $L$  that enforces the boundary conditions in equation (\ref{eq:nonlin}) 
 as 
\begin{equation}\label{eq:nonlin_lif}
P U_t + (A_i U)_{x_i}+A^T_i U_{x_i}+L(2\Sigma(B(U)-G))=0.
\end{equation}
The lifting operator exists only at the surface of the domain and for two smooth vector functions  $\phi, \psi$ it satisfies $\int\limits \phi^T   L(\psi) d \Omega = \oint\limits \phi^T  \psi ds$. As stated above, it enables development  of the numerical boundary procedure in the continuous setting \cite{Arnold20011749,nordstrom_roadmap}. Next, we prove the following Proposition.
\begin{proposition}\label{lemma:GenBC_sys}
Consider equation (\ref{eq:nonlin_lif}) with the general boundary operator $B(U)=C(U)U$ and penalty matrix $ \Sigma(U)$. The energy rate  is bounded by external boundary data, i.e. the boundary terms become
\begin{equation}\label{eq:BT}
U^T A^+U+ (\sqrt{ |A^-|}U-G)^T(\sqrt{ |A^-|}U-G)-G^TG  \geq - G^TG
\end{equation}
 if and only if $C=\sqrt{ |A^-|}$ and $\Sigma=\sqrt{ |A^-|}$.
\end{proposition}
\begin{proof} 
The energy method (multiply with $U^T$ and integrate over domain) applied to (\ref{eq:nonlin_lif}) yields 
\begin{equation}\label{eq:boundaryPart1_1_proof}
\frac{1}{2} \frac{d}{dt}\|U\|^2_P + \oint\limits_{\partial\Omega} U^T   \tilde A U+ 2U^T  \Sigma(B(U)-G)  ds= 0.
\end{equation}
By adding and subtracting $G^TG$  and assuming the general form $B(U)=C(U)U$ of the boundary operator, the argument in the boundary term on the lefthand side  in  (\ref{eq:boundaryPart1_1_proof}) becomes 
$U^T A^+ U-G^TG + \delta BT$, where 
\begin{equation}\label{eq:deltaBT_proof}
\delta BT=(\Sigma^T U-G)^T(\Sigma^T U-G) +U^T(-\sqrt{ |A^-|}\sqrt{ |A^-|}+\Sigma C +(\Sigma C)^T-\Sigma \Sigma^T)U.
\end{equation}
The second indefinite term vanish for $\Sigma^T=C=\sqrt{ |A^-|}$. Negative signs in $\Sigma,C$ can be absorbed in $G$.
\end{proof}
\begin{remark}\label{main_point} 
Proposition \ref{lemma:GenBC_sys} holds in both linear and nonlinear cases.
\end{remark}

\subsection{The general form of the new nonlinear charcteristic boundary conditions}\label{nonlinear_BC_char}
As often done for linear IBVPs, we generalize the new characteristic boundary condition $\sqrt{ |A^-|}  U-G=0$ and allow for contributions from the outgoing variables. The generalized boundary condition is 
\begin{equation}\label{Gen_BC_form_simp}
 B(U)-G=S^{-1}(\sqrt{ |A^-|}  - R \sqrt{ A^+} )U - G =0 \,\ \text{or equivalently} \,\  (\sqrt{ |A^-|}  - R \sqrt{ A^+} )U - SG =0,
\end{equation}
where $R$ and $S$ are appropriate matrices to be determined. 
The lifting operator in  (\ref{eq:nonlin_lif}) becomes
\begin{equation}\label{Pen_term_Gen_BC_form_simp}
L(2\Sigma(BU-G))= L(2 \sqrt{ |A^-|}   ( (\sqrt{ |A^-|} -R \sqrt{ A^+} ) U - SG )).
\end{equation}
The following Proposition with slightly modified notations is similar to Lemma 3.1 in \cite{NORDSTROM2024_BC}.
\begin{proposition}\label{lemma:GenBC}
Consider the governing equation (\ref{eq:nonlin_lif}), the boundary term $ 
U^T  A  U $, the boundary condition (\ref{Gen_BC_form_simp}) and the lifting operator $L$ in (\ref{Pen_term_Gen_BC_form_simp}).

The boundary term augmented with 1)  strong ($L=0$) nonlinear homogeneous boundary conditions is positive semi-definite i.e. 
\begin{equation}\label{BT1}
U^T  A  U   =  (\sqrt{ A^+} U)^T(I-R^T R)(\sqrt{ A^+}  U)  \geq 0 \nonumber
\end{equation}
if the matrix $R$ satisfies
\begin{equation}\label{R_condition}
I- R^T  R  \geq 0.
\end{equation}

The boundary term augmented with 2) strong ($L=0$) nonlinear inhomogeneous boundary conditions is bounded by the data $G$,  i.e. 
\begin{equation}
\label{BT2}
 U^T  A  U  =
\begin{bmatrix}
\sqrt{ A^+}  U \\
G
\end{bmatrix}^T
\begin{bmatrix}
I- R^T R & - R^T S \\
-S^T R  &  I-S^T S
\end{bmatrix}
\begin{bmatrix}
\sqrt{ A^+}  U \\
G
\end{bmatrix} - G^T G  \geq  -G^TG. 
\end{equation}
if the matrix  $R$ satisfies (\ref{R_condition}) with strict inequality and the matrix $S$ is such that
\begin{equation}\label{S_condition}
I- S^T  S - (R^T S)^T (I- R^T  R)^{-1} (R^T S)  \geq 0 \quad  \text{where}  \quad (I- R^T  R)^{-1}=\sum^{\infty}_{k=0}(R^T R)^k.
\end{equation}

The boundary term augmented with 3) weak ($L= L(2 \sqrt{ |A^-|}   ( (\sqrt{ |A^-|} -R \sqrt{ A^+} ) U ))$) nonlinear homogeneous boundary conditions is positive semi-definite i.e. the complete boundary term becomes
\begin{equation}\label{BT3}
 (\sqrt{ A^+}  U)^T(I-R^T R)(\sqrt{ A^+}  U) +  (\sqrt{ |A^-|}  U - R \sqrt{ A^+}  U)^T  (\sqrt{ |A^-|}  U - R \sqrt{ A^+}  U) \geq 0 \nonumber
\end{equation}
if the matrix $R$
satisfies (\ref{R_condition}).

The boundary term augmented with 4) weak ($L = L(2 \sqrt{ |A^-|}   ( (\sqrt{ |A^-|} -R \sqrt{ A^+} ) U - SG ))$) nonlinear inhomogeneous boundary conditions is bounded by the data $G$,  i.e. the complete boundary term becomes
\begin{align}
&  (\sqrt{ |A^-|} U -(R \sqrt{ A^+} U +SG))^T  (\sqrt{ |A^-|} U -(R \sqrt{ A^+} U +SG)) +  \nonumber \\ 
&  \begin{bmatrix}
\sqrt{ A^+} U \\
G
\end{bmatrix}^T
\begin{bmatrix}
I- R^T R & - R^T S \\
-S^T R  &  I-S^T S
\end{bmatrix}
\begin{bmatrix}
\sqrt{ A^+} U \\
G
\end{bmatrix} - G^TG \geq -G^TG \nonumber
\end{align}
if the matrix $R$ satisfies (\ref{R_condition}) with strict inequality and the matrix $S$ satisfies  (\ref{S_condition}).
\end{proposition}
\begin{remark}\label{explain_notation_nonstadard} 
The conditions  1-4  above are sufficient but not necessary conditions for limiting the boundary terms. We have used the notation $A \geq 0$ to indicate that the matrix $A$ is positive semi-definite above.
\end{remark}
\begin{proof} 
 The idea of the proof of Lemma \ref{lemma:GenBC} is to complete positive squares by adding and subtracting $G^TG$ (as was done in Proposition \ref{lemma:GenBC_sys}) and using specific choices of the matrix parameters $R$ and $S$. The proof is identical to the proof of Lemma 3.1 in \cite{NORDSTROM2024_BC} by replacing $\sqrt{|\Lambda^-|}W^-,\sqrt{\Lambda^+}W^+$,$\Sigma = 2( I^-T^{-1})^T \Sigma S$ with  $\sqrt{ |A^-|} U, \sqrt{ A^+} U$,$ \Sigma=2 \sqrt{ |A^-|} S$ respectively and $\tilde \Sigma = 2( I^-T^{-1})^T \sqrt{ |\Lambda^-|} S$ with $\Sigma=2 \sqrt{ |A^-|} S$.  
\end{proof}

\section{Summary and conclusions}\label{sec:conclusion}
We discussed similarities and differences between linear and nonlinear open boundary conditions. In particular, we compared the classical linear characteristic boundary conditions  with new nonlinear ones. It was shown that the classical characteristic boundary conditions do not lead to energy or entropy estimates while the new ones do. The main difference from the old linear characteristic conditions is that the new characteristic variables are scaled with the square root of the eigenvalues . The new nonlinear characteristic boundary conditions are stable for both for linear and nonlinear problems.
\section*{Acknowledgment}
Jan Nordstr\"{o}m was supported by Vetenskapsr{\aa}det, Sweden [award no.~2021-05484 VR] and University of Johannesburg Global
Excellence and Stature Initiative Funding.


\bibliographystyle{elsarticle-num}
\bibliography{References_Jan}

\end{document}